\documentclass[11pt]{amsart}

\usepackage{latexsym,amsmath,amssymb}
\usepackage[colorlinks=true]{hyperref}

\title[zero-sets of fractional Sobolev spaces]{\protect{A note on zero sets of fractional sobolev functions with negative power of integrability}}

\author{Armin Schikorra}

\address{Armin Schikorra, Max-Planck Institut MiS Leipzig, Inselstr. 22, 04103 Leipzig, Germany, {\tt armin.schikorra@mis.mpg.de}}

\thanks{A.S. was supported by DAAD fellowship D/12/40670}

plus 6pt minus 12pt
\def\eps{\varepsilon}

\newtheorem{theorem}{Theorem}
\newtheorem{lemma}[theorem]{Lemma}
\newtheorem{corollary}[theorem]{Corollary}

\theoremstyle{definition}

\newtheorem{definition}[theorem]{Definition}

\newcommand{\R}{\mathbb{R}}
\newcommand{\N}{\mathbb{N}}

\newcommand{\Z}{\mathbb{Z}}
\newcommand{\vrac}[1]{\| #1 \|}
\newcommand{\fracm}[1]{\frac{1}{#1}}
\newcommand{\brac}[1]{\left (#1 \right )}

\newcommand{\barint}{
\rule[.036in]{.12in}{.009in}\kern-.16in \displaystyle\int }

\newcommand{\barcal}{\mbox{$ \rule[.036in]{.11in}{.007in}\kern-.128in\int $}}

\def\bbbc{{\mathchoice {\setbox0=\hbox{$\displaystyle\rm C$}\hbox{\hbox
to0pt{\kern0.4\wd0\vrule height0.9\ht0\hss}\box0}}
{\setbox0=\hbox{$\textstyle\rm C$}\hbox{\hbox
to0pt{\kern0.4\wd0\vrule height0.9\ht0\hss}\box0}}
{\setbox0=\hbox{$\scriptstyle\rm C$}\hbox{\hbox
to0pt{\kern0.4\wd0\vrule height0.9\ht0\hss}\box0}}
{\setbox0=\hbox{$\scriptscriptstyle\rm C$}\hbox{\hbox
to0pt{\kern0.4\wd0\vrule height0.9\ht0\hss}\box0}}}}

\numberwithin{theorem}{section} \numberwithin{equation}{section}

\usepackage{graphicx}

\def\Xint#1{\mathchoice
{\XXint\displaystyle\textstyle{#1}}%
{\XXint\textstyle\scriptstyle{#1}}%
{\XXint\scriptstyle\scriptscriptstyle{#1}}%
{\XXint\scriptscriptstyle\scriptscriptstyle{#1}}%
\!\int}
\def\XXint#1#2#3{{\setbox0=\hbox{$#1{#2#3}{\int}$ }
\vcenter{\hbox{$#2#3$ }}\kern-.6\wd0}}
\DeclareMathOperator*{\mvint}{\Xint{--}}

\begin{document}
\sloppy

\subjclass[2000]{Primary 49Q15; Secondary 46E35}


\begin{abstract}
We extend a Poincar\'{e}-type inequality for functions with large zero-sets by Jiang and Lin to fractional Sobolev spaces. As a consequence, we obtain a Hausdorff dimension estimate on the size of zero sets for fractional Sobolev functions whose inverse is integrable. Also, for a suboptimal Hausdorff dimension estimate, we give a completely elementary proof based on a pointwise Poincar\'{e}-style inequality.
\end{abstract}
\maketitle

\section{Introduction}
Let $\Omega \subset \R^n$ be an open set. For functions $u: \Omega \to \R^n$ we are interested in the size of the zero set $\Sigma$,
\[
 \Sigma := \{x \in \Omega: \quad \lim_{r \to 0} \mvint_{B_r(x)} |f| = 0\}, 
\]
under the condition that for some $\alpha > 0$,
\begin{equation}\label{eq:negativeintegral}
 \int_\Omega |f|^{-\alpha} < \infty.
\end{equation}
Here and henceforth, for a measurable set $A \subset \R^n$ we denote the mean value integral
\[
 \mvint_{A} f \equiv (f)_A := |A|^{-1} \int_A f.
\]
In \cite{JiangLin} Jiang and Lin showed that if $f \in W^{1,p}(\Omega)$, then
\[
 \mathcal{H}^s(\Sigma) = 0 \quad \mbox{where $s = \max \{0,n-\frac{p\alpha}{p+\alpha} \}$}.
\]
They were motivated by the analysis of rupture sets of thin films, which is described by a singular elliptic equation. We do not go into the details of this and instead, for applications we refer to, e.g., \cite{RuptureDavilaPonce,RuptureZongmingSongbo,MR2898777,MR2754304}.

In this note, we extend Jiang and Lin's result to fractional Sobolev spaces and obtain
\begin{theorem}\label{th:thm}
For $\sigma \in (0,1]$ and for any $f \in W^{\sigma,p}(\Omega)$ satisfying \eqref{eq:negativeintegral}, $\mathcal{H}^s(\Sigma) = 0$, where $s = \max \{0,n-\sigma\frac{p\alpha}{p+\alpha} \}$.
\end{theorem}
Here, we use the following definitions for the (fractional) Sobolev space. For more on these we refer to, e.g., \cite{Hitchhiker,Adams03,SKM93}.
\begin{definition}
The homogeneous $W^{\sigma,p}$-norms are defined as follows:
\[
 [f]_{\dot{W}^{1,p}(\Omega)} := \vrac{\nabla f}_{L^p(\Omega)}.
\]
For $\sigma \in (0,1)$ we define the Slobodeckij-norm,
\[
 [f]_{\dot{W}^{\sigma,p}(\Omega)} := \begin{cases} 
\brac{\int\limits_\Omega \int\limits_\Omega \brac{\frac{|f(x)-f(y)|}{|x-y|^\sigma}}^p \frac{dx\ dy}{|x-y|^{n}}}^{\fracm{p}} \quad &\mbox{if } p \in [1,\infty),\\
                                                                                                      \\
                                                                                                      \sup\limits_{x \neq y} \frac{|f(x)-f(y)|}{|x-y|^\sigma} \quad &\mbox{if } p = \infty.
                                                                                                     \end{cases}
\]
The respective Sobolev space $W^{\sigma,p}$, $\sigma \in (0,1]$, $p \in [1,\infty]$ is then the collection of functions $f: \Omega \to \R$ with finite Sobolev norms $\vrac{f}_{W^{\alpha,p}(\Omega)}$,
\[
 \vrac{f}_{W^{\alpha,p}(\Omega)} := \vrac{f}_{L^p(\Omega)} + [f]_{\dot{W}^{\alpha,p}(\Omega)}.
\]
\end{definition}

To prove Theorem~\ref{th:thm}, the case $p \leq n/\sigma$ is the relevant one, since for the other cases we can use the embedding into the H\"older spaces, see \cite{JiangLin}. We have the following extension to fractional Sobolev spaces of a Poincar\'e-type inequality from \cite{JiangLin}.
\begin{theorem}\label{th:poincare}
For any $\theta > 0$, $\sigma \in (0,1]$, $p \in (1,n/\sigma]$, $s \in (n-\sigma p,n]$, there is a constant $C > 0$ such that the following holds for any $R > 0$:

Let $B_R$ be any ball in $\R^n$ with radius $R$, $f \in W^{\sigma,p}(B_R)$ and assume that there is a \emph{closed} set $T \subset B_R$ such that
\[
 T \subset \{x \in B_R: \quad \limsup_{r \to 0} \mvint_{B_r} |f| = 0 \},
\]
\begin{equation}\label{eq:largeness}
 \mathcal{H}^s(T) > \frac{1}{\theta}\ R^s,
\end{equation}
and for any ball $B_r$ with some radius $r > 0$,
\begin{equation}\label{eq:straightness}
 \mathcal{H}^s(T \cap B_r) \leq \theta r^s.
\end{equation}
Then,
\[
 \vrac{f}_{L^p(B_R)} \leq C\ R^\sigma\ [f]_{\dot{W}^{\sigma,p}(B_R)}.
\]
\end{theorem}
In \cite{JiangLin} this was proven for the classical Sobolev space $W^{1,p}$, using an argument based on the $p$-Laplace equation with measures and the Wolff potential. Our argument, on the other hand, is completely elementary and adapts the classical blow-up proof of the Poincar\'e inequality, see Section~\ref{s:poincproof}.

Once Theorem~\ref{th:poincare} is established, one can follow the arguments in \cite{JiangLin} to obtain Theorem~\ref{th:thm}. These rely heavily on the theory of Sousslin sets, \cite{Rogers}, to find the closed set $T \subset \Sigma$ with the condition \eqref{eq:largeness} and \eqref{eq:straightness} satisfied. Those arguments are by no means elementary, but we were unable to remove them in order to show that $\mathcal{H}^s(\Sigma) = 0$. However, if one is satisfied in showing that $\mathcal{H}^{t}(\Sigma) = 0$ for any $t > s$, then there is a completely elementary argument, the details of which we will present in Section~\ref{s:easyargs}. There, we prove the following ``pointwise'' Poincar\'e-style inequality, from which the suboptimal Hausdorff dimension estimate easily follows, see Corollary~\ref{co:subopt}.
\begin{lemma}\label{la:pointwisepoinc}
For any $\eps > 0$, $p \in [1,\infty)$, there exists $C > 0$, such that the following holds. Let $f \in L^p_{loc}$, and assume $x \in \R^n$, such that
\begin{equation}\label{eq:limtozero}
 \lim_{r \to 0} \mvint_{B_r(x)} |f| = 0
\end{equation}
then for any $R > 0$, there exists $\rho \in (0,R)$ such that
\[
 \int\limits_{B_\rho(x)} |f|^p \leq C\ \brac{\frac{R}{\rho}}^{\eps} \int\limits_{B_\rho(x)} ||f|-(|f|)_{B_\rho}|^p.  
\]
\end{lemma}

\subsection*{Acknowledgments}
The author thanks P. Haj\l{}asz for introducing him to Jiang and Lin's paper \cite{JiangLin}.

%

\section{Poincar\'e Inequality: Proof of Theorem~\ref{th:poincare}}\label{s:poincproof}
By a scaling argument, Theorem~\ref{th:poincare} follows from the following
\begin{lemma}
For any $\theta > 0$, $\sigma \in (0,1]$, $p \in (1,n/\sigma]$, $s \in (n-\sigma p,n]$, there is a constant $C > 0$ such that the following holds:

Let $f \in W^{\sigma,p}(B_1,[0,\infty))$ and assume that there is a \emph{closed} set $T \subset B_1$ such that
\[
 T \subset \{x \in B_1: \quad \limsup_{r \to 0} \mvint_{B_r} f = 0 \},
\]
and
\[
 \mathcal{H}^s(T) > \frac{1}{\theta},
\]
as well as
\[
 \mathcal{H}^s(T \cap B_r) \leq \theta r^s \quad \mbox{for any ball $B_r$ with radius $r > 0$}.
\]
Then,
\[
 \vrac{f}_{L^p(B_1)} \leq C\ [f]_{\dot{W}^{\sigma,p}(B_1)}.
\]
\end{lemma}
\begin{proof}
We proceed by the usual blow-up proof of the Poincar\'e inequality: Assume the claim is false, and that for fixed $\theta, p, s, \sigma$ for any $k \in \N$ there are $f_k \in W^{\sigma,p}(B_1,[0,\infty))$ such that
\[
 T_k \subset \{x \in B_1: \quad \limsup_{r \to 0} \mvint_{B_r} f_k = 0 \},\]
\[
 \mathcal{H}^s(T_k) > \frac{1}{\theta},\quad \mathcal{H}^s(T_k \cap B_r) \leq \theta r^s\ \forall B_r,
\]
and
\[
 \vrac{f_k}_{L^p(B_1)} > k\ [f_k]_{\dot{W}^{\sigma,p}(B_1)}.
\]
Replacing $f_k$ by $\frac{f_k}{\vrac{f_k}_p}$ (note that this does not change the definition and size of $T_k$), we can assume w.l.o.g.
\[
 \vrac{f_k}_{L^p} \equiv 1,
\]
and
\[
 [f_k]_{\dot{W}^{\sigma,p}(B_1)} \xrightarrow{k \to \infty} 0.
\]
In particular, $f_k$ is uniformly bounded in $W^{\sigma,p}$, and by the Rellich-Kondrachov theorem, up to taking a subsequence, $f_k$ converges strongly in $L^p$, and weakly in $W^{\sigma,p}$ to some $f \in W^{\sigma,p}$, with $ [f]_{\dot{W}^{\sigma,p}(B_1)} \equiv 0$, $\vrac{f}_{L^p} = 1$. Thus,
\[
 f \equiv |B_1|^{-\frac{1}{p}},
\]
and setting $g_k := |B_1|^{\frac{1}{p}} f_k$, we have found a sequence such that
\[
 g_k \to 1 \quad \mbox{in $W^{\sigma,p}(B_1)$},
\]
\[
 \mathcal{H}^s(T_k) > \frac{1}{\theta}, 
\]
and
\[
\mathcal{H}^s(T_k \cap B_r) \leq \theta r^s \quad \mbox{for any ball $B_r$}.
\]
This is a contradiction to Lemma~\ref{la:noconvergence}.
\end{proof}

We used the following lemma, which essentially quantifies the intuition, that a function approximating $1$ in $W^{\sigma,p}$ cannot be zero on a large set.
\begin{lemma}\label{la:noconvergence}
Let $\sigma \in (0,1]$, $s \in (n-\sigma p,n]$, $f_k \in W^{\sigma,p}(B_1,[0,\infty))$, and assume that
\[
 \vrac{f_k - 1}_{W^{\sigma,p}(B_1)} \xrightarrow{k \to \infty} 0.
\]
Then, for any $T_k \subset B_1$ closed and
\[
 T_k \subset \{x \in B_1: \quad \limsup_{r \to 0} \mvint_{B_r} f_k = 0 \},
\]
as well as for some $\theta > 0$,
\begin{equation}\label{eq:Hsupperbound}
 \mathcal{H}^s(T_k \cap B_r) \leq \theta r^s \quad \mbox{for any $B_r$, for all $k$}
\end{equation}
we have
\[
 \lim_{k \to \infty} \mathcal{H}^s(T_k) =0.
\]
%
\end{lemma}
\begin{proof}
By the subsequence principle, it suffices to show 
\[
  \liminf_{k \to \infty} \mathcal{H}^s(T_k) =0.
\]
By extension, we also can assume that $f_k -1\to 0$ in $W^{\sigma,p}(\R^n)$, and $f_k \equiv 1$ on $\R^n \backslash B_2$.

On the one hand, we have
\[
  [f_k]_{\dot{W}^{\sigma,p}(\R^n)} \xrightarrow{k \to \infty} 0.
\]
On the other hand, up to picking a subsequence, we can assume the existence of $R_k \in (0,1)$, for $k \in \N$, and $\lim_{k \to \infty} R_k = 0$, such that
\[
 \inf_{r > R_k, x \in B_1} \mvint\limits_{B_{r}(x)} f_k \geq \frac{9}{10}.
\]
Since for any point $x \in T_k$ we have that $\lim_{t \to 0}\mvint_{B_r} f_k(x) = 0$, we expect the the average (fractional) gradient around $x$ to be fairly large. More precisely, we have the following

\subsection*{Claim}
There is a uniform constant $c_{s,\sigma,p} > 0$, such that the following holds: For any $x \in T_k$, there exists $\rho = \rho_{k,x} \in (0,R_k)$ such that 
\begin{equation}\label{eq:nablafpgeqcsrhos}
c_{s,\sigma,p}\ \rho^{s} \leq \rho^{-\sigma p} \int_{B_{\rho}} |f_k-(f_k)_{B_\rho}|^p \leq C\ [f_k]_{\dot{W}^{\sigma,p}(B_\rho)}^p.
\end{equation}
Of course, we only have to show the first inequality, the second inequality is the classical Poincar\'e inequality.

For the proof let us write $f$ instead of $f_k$. Then, since for $x \in T$,
\[
 \lim_{l \to \infty} \mvint_{B_{2^{-l-1} R_k(x)}} f = 0,
\]
we have that
\begin{align*}
 \frac{9}{10} &\leq \sum_{l = 0}^\infty \brac{ \mvint_{B_{2^{-l} R_k}(x)} f - \mvint_{B_{2^{-l-1} R_k(x)}} f}\\
&\leq C\ \sum_{l = 0}^\infty \brac{(2^{-l} R_k)^{-n} \int_{B_{2^{-l} R_k}} |f - (f)_{B_{2^{-l} R_k}}|}.
\end{align*}
Consequently, for any $\eps > 0$, there has to be some $c_\eps > 0$ and some $l \in \N$ such that
\[
\brac{(2^{-l} R_k)^{-n} \int_{B_{2^{-l} R_k}} |f - (f)_{B_{2^{-l} R_k}}|} \geq c_\eps \brac{2^{-l}R_k}^\eps,
\]
because if the opposite inequality was true for all $l \in \N$ we would have
\[
 \frac{9}{10} \leq C\ c_\eps R_k^\eps \sum_{l \in \N} 2^{-\eps l} \leq C\ c_\eps  \sum_{l \in \N} 2^{-\eps l}.
\]
which is false for $c_\eps$ small enough.

Thus, for $\rho := 2^{-l} R_k \in (0,R_k)$,
\[
 \rho^{n-\sigma +\eps} \leq C_\eps \rho^{-\sigma} \int_{B_{\rho}} |f-(f)_{B_\rho}| \leq C_\eps\ \brac{\rho^{-\sigma p} \int_{B_{\rho}} |f-(f)_{B_\rho}|^p}^{\frac{1}{p}} \rho^{n-\frac{n}{p}},
\]
that is
\[
 \rho^{n-\sigma p+\eps p} \leq C_\eps\ \rho^{-\sigma p} \int_{B_{\rho}} |f-(f)_{B_\rho}|^p,
\]
Setting $\eps = \frac{s-(n-\sigma p)}{p} > 0$, we have shown for any $x \in T$ the existence of some $\rho \in (0,R_k)$ satisfying \eqref{eq:nablafpgeqcsrhos}, and the claim is proven.


For any $k$ we cover $T_k$ by the family
\[
 \mathcal{F}_k := \{B_\rho(x), \quad x \in T,\ B_\rho(x) \mbox{ satisfies \eqref{eq:nablafpgeqcsrhos}}\}.
\]
Since $T \subset B_2$ is closed and bounded, i.e. compact, we can find a finite subfamily still covering all of $T_k$, and then using Vitali's (finite) covering theorem, we find a subfamily $\tilde{\mathcal{F}}_k \subset \mathcal{F}_k$ of disjoint balls $B_\rho(x)$, so that the union of the $B_{5\rho}$ covers all of $T_k$. We use this $\tilde{\mathcal{F}}_k$ as a cover for an estimate of the Hausdorff measure:
\begin{align*}
\mathcal{H}^s(T_k) &\leq \sum_{B_{\rho} \in \tilde{\mathcal{F}}_k} \mathcal{H}^s(B_{5\rho} \cap T_k)\overset{\eqref{eq:Hsupperbound}}{\leq} \theta\ 5^s\ \sum_{B_{\rho} \in \tilde{\mathcal{F}}_k} \rho^s\\
  &\overset{\eqref{eq:nablafpgeqcsrhos}}\leq C_{\theta,s} \sum_{B_{\rho} \in \tilde{\mathcal{F}}_k} [f_k]_{\dot{W}^{\sigma,p}(B_\rho)}^p \leq C_{\theta,s}\ [f_k]_{\dot{W}^{\sigma,p}(\R^n)}^p 
  \xrightarrow{k \to \infty} 0.
\end{align*}

%
\end{proof}

\section{An elementary proof for the suboptimal case}
\label{s:easyargs}
We start with the proof of the pointwise inequality, Lemma~\ref{la:pointwisepoinc}.
\begin{proof}
First, let us show the claim for  $p = 1$:

Fix $R,\eps > 0$, $f \in L^1_{loc}$ and assume $x = 0$. W.l.o.g., $f \geq 0$. Set
\begin{equation}\label{eq:tausmall}
 \tau = 2^{-n-1} \brac{\sum_{l=-\infty}^0 2^{\eps l}}^{-1} R^{-\eps},
\end{equation}
and $C_\eps := R^{-\eps} \tau^{-1}$. Assume by contradiction that the claim was false, i.e. assume that for any $\rho \in (0,R)$, 
\begin{equation}\label{eq:assumptionbycontra}
 \mvint_{B_\rho} |f-(f)_{B_\rho}| < \tau\ \rho^\eps\ \mvint_{B_\rho} f.
\end{equation}
Then for any $K \in \N$,
\begin{align*}
 \mvint_{B_\rho} |f-(f)_{B_\rho}| &< \tau\ \rho^\eps\  \sum_{k=-K}^0 \mvint_{B_{2^k\rho}} f - \mvint_{B_{2^{k-1}\rho}} f \quad + \tau \rho^\eps\ \mvint_{B_{2^{-K-1}\rho}} f\\
 &\leq 2^n \tau\ \rho^\eps\ \sum_{k=-K}^0 \mvint_{B_{2^k \rho}} |f-(f)_{B_{2^k \rho}}|  + \tau \rho^\eps\ \mvint_{B_{2^{-K-1}\rho}} f
\end{align*}
Setting now for $l \in \Z$,
\[
 a_l := \mvint_{B_{2^l R}} |f-(f)_{B_{2^l R}}|,
\]
\[
 b_l := \mvint_{B_{2^l R}} f,
\]
the above equation applied to $\rho = 2^l R$ reads as
\[
 a_l \leq 2^n R^\eps\ \tau\ 2^{\eps l}\ \sum_{k=-K}^0 a_{k+l} + \tau\ (2^{l} R)^{\eps}\ b_{-K+l-1} \quad \mbox{for any $K \in \N$, $l \in -\N$}.	
\]
In particular for any $L \in \N$, 
\begin{align*}
 \sum_{l=-L}^0 a_l &\leq 2^n R^\eps\ \tau\ \sum_{l=-L}^0 2^{\eps l}\ \sum_{k=-K}^0 a_{k+l} +  \tau\ R^{\eps}\ \sum_{l=-L}^0 2^{\eps l}\ b_{-K+l-1}\\
&\leq 2^n R^\eps\ \tau\ \sum_{l=-L}^0 2^{\eps l}\ \sum_{k=-K+l}^0 a_{k} +  \tau\ R^{\eps}\ (\sup_{j \leq -K} b_{j})\ \sum_{l=-\infty}^0 2^{\eps l} \\
&\leq 2^n R^\eps\ \tau\ \sum_{k=-L-K}^0 a_{k}\ \sum_{l=-L}^{k+K} 2^{\eps l}  +  \tau\ R^{\eps}\ (\sup_{j \leq -K} b_{j})\ \sum_{l=-\infty}^0 2^{\eps l} \\
&\overset{\eqref{eq:tausmall}}{\leq} \frac{1}{2} \sum_{k=-L-K}^0 a_{k}  +  \frac{1}{2} \sup_{j \leq -K} b_{j}.
\end{align*}
Under the additional assumption that
\begin{equation}\label{eq:finitesum}
 \sum_{l=-\infty}^0 a_l < \infty,
\end{equation}
letting $L,K \to \infty$, using that by \eqref{eq:limtozero} we have $\lim_{l \to \infty} b_l = 0$, the above estimates implies that $a_k = 0$ for all $k \leq 0$. This means that $f$ is a constant on $B_R$, and in particular by \eqref{eq:limtozero}, $f$ is constantly zero in $B_R$. This contradicts the strict inequality \eqref{eq:assumptionbycontra}.

To see \eqref{eq:finitesum}, fix $K \in \N$ such that $\sup_{j \leq -K} b_{j} \leq 2$. Then for
\[
 c_L := \sum_{l=-L}^0 a_l,
\]
the above estimate becomes
\[
 c_{L} \leq \frac{1}{2} c_{L+K} + 1 \quad \mbox{for any $L \in \N$}.
\]
In particular, for any $i \in \N$,
\[
 c_{L+iK} \leq 2^{-i} c_{L} + \sum_{j=0}^i 2^{-j}.
\]
Since $c_{i}$ is monotonically increasing,
\[
 \sup_{i \geq L+K} c_{i} \leq c_{L} + \sum_{j=0}^\infty 2^{-j} < \infty.
\]
This proves Lemma~\ref{la:pointwisepoinc} for $p =1$. 

If $p > 1$, we apply this to $f^p$, and obtain
\begin{equation}\label{eq:firstfpineq}
\int\limits_{B_\rho(x)} f^p \leq C\ \brac{\frac{R}{\rho}}^{\eps} \int\limits_{B_\rho(x)} |f^p-(f^p)_{B_\rho}|.
\end{equation}
We now need the following estimate, which holds for any $p \in [1,\infty)$, and $\delta \in (0,1)$,
\[
 \big ||a-b|^p - |a|^p - |b|^p \big | \leq \delta |a|^p + \frac{C_p}{\delta^p} |b|^p.
\]
Since $B_\rho$ is fixed, let us write $(f)$ for $(f)_{B_\rho}$. Firstly, for any $\delta \in (0,1)$,
\[
 \big |f^p-(f^p) \big | \leq \big |f-(f) \big |^p + \big | (f)^p-(f^p) \big | + \frac{C}{\delta^p} |f-(f)|^p + \delta (f)^p.
\]
Plugging this in \eqref{eq:firstfpineq}, for $\delta = \tilde{\delta} (R/\rho)^{-\eps}$ small enough, we arrive at
\begin{equation}\label{eq:2ndfpineq}
\int\limits_{B_\rho(x)} f^p \leq C\ \brac{\frac{R}{\rho}}^{(1+p)\eps} \int\limits_{B_\rho(x)} |f-(f)|^p + C\ \rho^n \brac{\frac{R}{\rho}}^{(1+p)\eps} \big | (f)^p-(f^p) \big |.
\end{equation}
Next,
\[
 \big | (f)^p-(f^p) \big | \leq \big (| (f)^p-f^p \big |) \leq (|f-(f)|^p) + \delta f^p + \frac{C}{\delta^p} (|f-(f)|^p).
\]
Plugging this now for $\delta = \tilde{\delta} (R/\rho)^{-(1+p)\eps}$ into \eqref{eq:2ndfpineq}, by absorbing we arrive at
\[
\int\limits_{B_\rho(x)} f^p \leq C\ \brac{\frac{R}{\rho}}^{\eps c_p} \int\limits_{B_\rho(x)} |f-(f)|^p.
\]
Since this holds for $\eps > 0$ is arbitrarily small, this proves the Lemma~\ref{la:pointwisepoinc}.
\end{proof}

\begin{corollary}\label{co:subopt}
For $\sigma \in (0,1]$ and for any $f \in W^{\sigma,p}(\Omega)$ satisfying \eqref{eq:negativeintegral}, $\mathcal{H}^t(\Sigma) = 0$, whenever $t > s = \max \{0,n-\sigma\frac{p\alpha}{p+\alpha} \}$.
\end{corollary}
\begin{proof}
Let $\eps > 0$, $R > 0$, and $x \in \Sigma$. Pick $\rho < R$ from Lemma~\ref{la:pointwisepoinc}, so that
\[
 \int\limits_{B_\rho(x)} |f|^p \leq C\ R^\eps \rho^{\sigma p-\eps}\ [f]_{\dot{W}^{\sigma,p}(B_\rho)}^p.
\]
By H\"older and Young inequality, as in \cite[Corollary 2.1]{JiangLin},
\begin{align*}
 \rho^{n+(2\eps - \sigma p)\frac{\alpha}{p+\alpha}} &\leq C\ \rho^{2\eps - \sigma p}\int\limits_{B_\rho(x)} |f|^p + C\rho^\eps \int\limits_{B_\rho(x)} |f|^{-\alpha}\\
 &\leq C\ R^{2\eps} [f]_{\dot{W}^{\sigma,p}(B_\rho)}^p + C\ R^\eps \int\limits_{B_\rho(x)} |f|^{-\alpha}.
\end{align*}
Let now $\eps > 0$ such that $t > n+(2\eps - \sigma p)\frac{\alpha}{p+\alpha}$, then what we have shown is that for any $R > 0$ and any $x \in \Sigma$ there exists $\rho \in (0,R)$ such that
\begin{equation}\label{eq:theestimate}
  \rho^{t} \leq C\ R^\eps [f]_{\dot{W}^{\sigma,p}(B_\rho)}^p + C \int\limits_{B_\rho(x)} |f|^{-\alpha}.
\end{equation}
Let now 
\[
 \mathcal{V}_R := \{B_\rho(x):\ x \in \Sigma,\ \rho < R, \eqref{eq:theestimate}\mbox{ holds}\}.
\]
Any countable disjoint subclass $\mathcal{U}_R \subset \mathcal{V}_R$ satisfies
\[
 \sum_{B_\rho \subset \mathcal{U}_R} \rho^t \leq C\ R^\eps [f]_{\dot{W}^{\sigma,p}(\Omega)}^p + C R^\eps \int\limits_{\Omega} |f|^{-\alpha}.
\]
By the Besicovitch covering theorem, as in, e.g., \cite[Theorem 18.1]{DBRealAnalysis}, we find for any $R$ a countable subclass $\mathcal{U}_R \subset \mathcal{V}_R$, such that any point of $\Sigma$ is covered at least once, and at most a fixed number of times. Thus, 
\[
 \mathcal{H}^t(\Sigma) = \lim_{R \to 0} \mathcal{H}^t_{R}(\Sigma) \leq C\ \lim_{R \to 0} \sum_{B_\rho \subset \mathcal{U}_R} \rho^t \leq C_f \lim_{R \to 0} R^\eps = 0.
\]
\end{proof}

\bibliographystyle{plain}%
\bibliography{bib}%
\end{document}